\documentclass[11pt, a4paper]{amsart}

\usepackage{amsfonts,amsmath,amssymb,amsthm}

\newcommand{\Nn}{\mathbb{N}}
\newcommand{\Zz}{\mathbb{Z}}
\newcommand{\Qq}{\mathbb{Q}} 
\newcommand{\Ff}{\mathbb{F}}

\renewcommand{\epsilon}{\varepsilon}
\renewcommand{\le}{\leqslant}

\renewcommand{\leq}{\leqslant}
\renewcommand{\geq}{\geqslant}

{\theoremstyle{plain}
\newtheorem{theorem}{Theorem}[section]    
\newtheorem{lemma}[theorem]{Lemma}       
\newtheorem{proposition}[theorem]{Proposition}  
\newtheorem{corollary}[theorem]{Corollary}   
}
{\theoremstyle{remark}
\newtheorem{remark}[theorem]{Remark}   

}

\begin{document}

\baselineskip=17pt 

\title{Specializations of indecomposable polynomials}

\author{Arnaud Bodin}
\email{Arnaud.Bodin@math.univ-lille1.fr}

\author{Guillaume Ch\`eze}
\email{guillaume.cheze@math.univ-toulouse.fr}

\author{Pierre D\`ebes}
\email{Pierre.Debes@math.univ-lille1.fr}

\address{Laboratoire Paul Painlev\'e, Math\'ematiques, Universit\'e 
Lille 1, 59655 Villeneuve d'Ascq Cedex, France}

\address{Institut de Math\'ematiques de Toulouse, Universit\'e 
Paul Sabatier Toulouse 3, 31062 Toulouse Cedex 9, France}

\address{Laboratoire Paul Painlev\'e, Math\'ematiques, Universit\'e 
Lille 1, 59655 Villeneuve d'Ascq Cedex, France}



\date{\today}

\begin{abstract}
We address some questions concerning indecomposable 
polynomials and their behaviour under specialization. 
For instance we give a bound on a prime $p$ for the reduction modulo $p$ of an indecomposable polynomial $P(x)\in \Zz[x]$ to remain indecomposable. We also obtain a Hilbert like result for indecomposability: if $f(t_1,\ldots,t_r,x)$ is an indecomposable polynomial in several variables with coefficients in a field of characteristic $p=0$ or $p>\deg(f)$, then the one variable specialized 
polynomial $f(t_1^\ast+\alpha_1^\ast x,\ldots,t_r^\ast+\alpha_r^\ast x,x)$ is indecomposable for all 
$(t_1^\ast, \ldots, t_r^\ast, \alpha_1^\ast, \ldots,\alpha_r^\ast)\in \overline k^{2r}$ off a proper Zariski closed subset.
\end{abstract}

\maketitle

\section{Introduction} 
\label{sec:intro}

Let $x$ be an indeterminate. A non-constant polynomial $f(x)\in k[x]$ with coefficients in a field $k$ is said to be {\it decomposable} in $k[x]$ if it is of the form $u(g(x))$ with $g$ \underbar{and} $u$ in $k[x]$ of degree $\geq 2$, and {\it indecomposable} otherwise. 
For polynomials in several variables, the definition is slightly different: for an integer $n\geq 2$ and a $n$-tuple $\underline x=(x_1,\ldots,x_n)$ of indeterminates, a non-constant polynomial $f(\underline x)\in k[\underline x]$ is  {\it decomposable} in $k[\underline x]$ if it is of the form $u(g(\underline x))$ with $u\in k[t]$ of degree $\geq 2$ and $g(\underline x)\in k[\underline x]$; unlike for $n=1$, the case $\deg(g)=1$, $\deg(u)\geq 2$ is allowed. 

The central theme of the paper is the following general problem. Let $A$ be an integral domain with quotient field $K$ and  
$f(x)\in A[x] $ be an indecomposable polynomial in $K[x]$. Given a ring morphism 
$\sigma : A\rightarrow k$ with $k$ 
a field,  the question is whether the polynomial $f^\sigma(x)$ obtained by applying $\sigma$ to the coefficients of $f(x)$ is also indecomposable. 

We have a first general statement {\it \`a la Bertini-Noether} under the assumption 
that $\deg(f)$ is prime to the characteristic of $K$\footnote{{\it i.e.} $\hbox{\rm char}(K)$ equals $0$ or does not  divide $\deg(f)$.}: the answer is positive ``generically'', that is, for all $\sigma$ such that $I_f^\sigma \not=0$ where $I_f$ is some non-zero element of $A$ depending only on $f$
(proposition \ref{prop:morphism}). Based on a general decomposition result 
for polynomials, established in \cite{Bo}, our approach leads to quite explicit versions of the Bertini-Noether conclusion. For polynomials in several variables, similar conclusions had already been
proved (see \cite{BDN}, \cite{BCN}, \cite{ChezeNajib}); the single variable case is somewhat different.

We investigate further two typical situations. The first one is for $A=\Zz$ and $\sigma: \Zz \rightarrow {\Ff_p}$ a reduction morphism modulo $p$. The Bertini-Noether conclusion  
is here that $f^\sigma(x)$ is indecomposable if $p$ is suitably large. 
Our method leads to the following explicit version. To our knowledge no such 
bound as the one below was previously available. 

\begin{theorem} \label{thm:reduction}
Let $f(x) \in \Zz[x]$ be indecomposable in $\Qq[x]$. There exists a constant $\gamma_d$ depending only on $d = \deg(f)$ such that if $p > \gamma_d\hskip 2pt  \| f \|_\infty^d$ is a prime, the reduced polynomial $\overline f(x)$ modulo $p$ is indecomposable in $\overline \Ff_p[x]$.
\end{theorem}

We then focus on the situation where $A=k[\underline t]$ with $\underline t = (t_1,\ldots,t_r)$ an $r$-tuple of indeterminates ($r\geq 1$), $k$ a field and $\sigma: k[\underline t] \rightarrow \overline k$ a specialization morphism sending each $t_i$ to a special value $t_i^\ast \in \overline k$, $i=1,\ldots,r$. In this situation, the Bertini-Noether conclusion is that if $f(\underline t, x)\in  k[\underline t,x]$ is indecomposable in $k(\underline t)[x]$ and of degree prime to the characteristic of $k$, then for all $\underline t^\ast =(t_1^\ast, \ldots, t_r^\ast) \in k^r$ but in a proper Zariski closed subset, the specialized polynomial $f(t_1^\ast,\ldots,t_r^\ast, x)$ is indecomposable in $k[x]$. 

The indecomposability assumption excludes polynomials $f$ of the form $u(\underline t,g(\underline t, x))$ with $u,g\in k[\underline t,x]$. 
It is natural to ask whether the Bertini-Noether conclusion extends to such polynomials 
and more generally to all polynomials that are indecomposable in $k[\underline t,x]$ ({as $(r+1)$-variable polynomials}). 
Although this is not true in general (take for example $f(t,x)=tx^4$),
we show nevertheless that the desired conclusion does hold up to some change of variables. Specifically we obtain the following result.

\begin{theorem} \label{th:main}
Let $f(\underline t, x)$ be indecomposable in $k[\underline t, x]$. Assume that $k$ is of characteristic $p=0$ or $p> \deg(f)$. Then we have the following:
\vskip 1mm

\noindent
{\rm (a)} if $\underline \alpha = (\alpha_1,\ldots, \alpha_r)$ is an $r$-tuple of indeterminates, the polynomial \hfill \break $f(t_1+\alpha_1x,\ldots,t_r+\alpha_rx, x)$ is indecomposable in $\overline {k(\underline \alpha, \underline t)}[x]$; \vskip 1mm

\noindent
{\rm (b)} for all $(\alpha_1^\ast,\ldots,\alpha_r^\ast )\in \overline k^{r}$ off a proper Zariski closed subset, the polynomial $f(t_1+\alpha_1^\ast x,\ldots,t_r+\alpha_r^\ast x, x)$ is indecomposable in $\overline {k(\underline t)}[x]$;  
\vskip 1mm

\noindent
{\rm (c)} for all $(\alpha_1^\ast,\ldots,\alpha_r^\ast,t_1^\ast,\ldots,t_r^\ast) \in \overline k^{2r}$ off a proper Zariski closed subset, the polynomial $f(t_1^\ast +\alpha_1^\ast x,\ldots,t_r^\ast +\alpha_r^\ast x, x)$ is indecomposable in $\overline k[x]$.
\vskip 1mm
\end{theorem}

Combined with the standard fact that $f(\underline t, x)$ is indecomposable in $\overline k[\underline t, x]$ if it is irreducible in $\overline k[\underline t,x]$, theorem \ref{th:main} has the 
following consequence which makes it easy to produce indecomposable polynomials in one variable.

\begin{corollary} \label{cor:intro} Let $f(\underline t, x)$ be irreducible in $\overline k[\underline t, x]$. Assume that $k$ is of characteristic $p=0$ or $p> \deg(f)$. Then for all $(\alpha_1^\ast,\ldots,\alpha_r^\ast,t_1^\ast,\ldots,t_r^\ast) \in \overline k^{2r}$ off a proper Zariski closed subset, the polynomial $f(t_1^\ast +\alpha_1^\ast x,\ldots,t_r^\ast +\alpha_r^\ast x, x)$ is indecomposable in $\overline k[x]$. 
\end{corollary}

The assumption on the characteristic of $k$ in theorem \ref{th:main} guarantees that $f(\underline t, x)$ is 
indecomposable in $\overline k[\underline t, x]$ under the condition that it is indecomposable in $k[\underline t, x]$. 
This follows from \cite[theorem 4.2]{BDN}. A similar result holds for polynomials in one variable
\cite[lemma 21.8.11]{FrJa}. We will use these results in several occasions. We will further show that
this assumption on the characteristic of $k$ cannot be removed in theorem \ref{th:main} (see remark \ref{rem:contre-exemple}).

Theorem \ref{th:main} and corollary \ref{cor:intro} can be made more explicit: for two variables polynomials ($r=1$), the exceptional Zariski closed subset has a  degree bounded by  $\deg(f)^3+2\deg(f)$, see corollary \ref{thm:optimal_bound_new}.

A main step in theorem \ref{th:main} is to go from two to one variable (that is, the case $r=1$). A key ingredient is a partial differential equation satisfied by the roots of a polynomial equation (Burger's equation lemma \ref{lem:burger}) due to Wood \cite{Wo} and investigated further by Lecerf and Galligo \cite{GalLec}.

\medskip

\noindent
{\bf Acknowledgments.} 
The second author thanks G. Lecerf and A. Galligo for interesting discussions about Burger's equation.


\section{Preliminaries and first results}

\subsection{Decomposition of polynomials} 

A useful tool is the following decomposition result for polynomials in one variable,
established in \cite{Bo}.

Let $A$ be an integral domain, $f\in A[x]$ be a monic polynomial of degree $d$ and $m\geq 2$ be a divisor of $d$ that is invertible in $A$. Then there exists a unique triple $(u,g,h)$ of polynomials in $A[x]$ such that
\smallskip

\noindent
$(m{\rm -dec})$ \hskip 3cm $f(x) =u(g(x)) + h(x)$
\smallskip

\noindent
with the conditions that

\noindent
{\rm (i)} $u$ and $g$ are monic,

\noindent
{\rm (ii)} $\deg(u) = m$, the coefficient of $x^{m-1}$ in $u$ is $0$ and $\displaystyle \deg(h) < d - \frac{d}{m}$,

\noindent
{\rm (iii)} $h(x) = \sum_i h_i x^i$ with $(\deg(g) \hskip 2pt |\hskip 2pt  i \Rightarrow h_i = 0)$.
\smallskip

In particular, if $A$ is a field and  $1< m <d$, $f(x)$ is $m$-decomposable in $A[x]$ ({\it i.e.} decomposable with the polynomial $u$ from the definition of degree $m$) if and only if $h(x)=0$ 
in the above $m$-decomposition.

\begin{remark} Using this decomposition, one easily deduces the following statement which can be compared to theorem 2 of \cite{AN}:
\smallskip

\noindent
{\it Let $f(t_1,\ldots,t_r,x) \in k[t_1,\ldots,t_r][x]$ be a monic polynomial with $\deg(f) = d$ prime to the characteristic of $k$ and $m$ be a divisor of $d$ with $1<m<d$. If $f$ is $m$-decomposable in $k[t_1,\ldots,t_r][x]$ then for all $f^\prime(t_1,\ldots,t_r) \in k[t_1,\ldots,t_r]\setminus k$, $f + f^\prime$ is $m$-indecomposable  in $k[t_1,\ldots,t_r][x]$. }
\smallskip

\noindent
Indeed assume $f = u(g)$ with $u \in k[x]$ of degree $m$ and $g \in k[t_1,\ldots,t_r][x]$. Deduce that the $m$-decomposition (with $A=k[t_1,\ldots,t_r]$) of $f+f^\prime$ is $f+f^\prime = u^\prime(g)$ with 
$u^\prime(x) = u(x)+f^\prime$ (and no remainder). As $u^\prime\notin k[x]$, conclude with proposition 7 from \cite{Bo} that $f+f^\prime$ is not $m$-decomposable in $k[t_1,\ldots,t_r][x]$.
\end{remark} 

Next we recall from \cite{Bo} this more technical information on the decomposition ($m$-dec) that we will use later: the polynomial $g(x)$ is the  {\it approximate $m$-root of $f(x)$}. More specifically if $f(x) = x^d+a_1x^{d-1} +\cdots + a_d$ and $g(x) = x^{\frac{d}{m}}+b_1x^{{\frac{d}{m}}-1} +\cdots + b_{{\frac{d}{m}}}$,  we have

\begin{equation*}
\label{eq:sys}
\begin{cases}
  a_1 = m b_1 \\
  a_2 = m b_2+ \binom{m}{2} \hskip 2pt b_1^2 \\
  \vdots \\
  a_i = m b_i + {\hspace*{-1em} \displaystyle{\sum_{\substack{j_1+2j_2+\cdots+(i-1)j_{i-1}=i \\ j_1+j_2+\cdots + j_{i-1} \le m}}\hspace*{-1em} c_{j_1\ldots j_{i-1}} b_1^{j_1}\cdots b_{i-1}^{j_{i-1}}  }}, \qquad 1 \le i \le \frac{d}{m}
\end{cases}
\tag{$\mathcal{S}$}
\end{equation*}
where the coefficients $c_{j_1\ldots j_{i-1}}$ are the multinomial coefficients defined by the following formula:
$$c_{j_1\ldots j_{i-1}} = \binom{m}{j_1,\ldots,j_{i-1}} = \frac{m!}{j_1! \cdots j_{i-1}!(m-j_1-\cdots-j_{i-1})!}.$$

Once $g$ has been obtained we get the full decomposition as follows: first compute $f^{(1)} = f-g^m$ and set
$u^{(1)} =x^d$, $h^{(1)} = 0$.
If for the highest monomial $\alpha x^i$ of $f^{(1)}$, $i$ is divisible by $\frac dm$ then set
$f^{(2)} = f^{(1)} - \alpha g^{i\frac md}$, $u^{(2)} = u^{(1)} + \alpha x^{i\frac md}$ and $h^{(2)} = h^{(1)}$ ;
if $i$ is not divisible by $\frac dm$ then set $f^{(2)} = f^{(1)} - \alpha x^i$, $u^{(2)} = u^{(1)}$ and 
$h^{(2)} = h^{(1)} + \alpha x^i$. Then iterate the process with $f^{(1)}$ replaced by $f^{(2)}$.


\subsection{Further degree estimates for polynomials in two variables}
Let $A$ be an integral domain and $f(t,x) \in A[t,x]$ of degree $d$, monic in $x$:
$$f(t,x) = x^d + a_1(t)x^{d-1}+\cdots+a_d(t)$$
with $\deg a_i(t) \le i$, $1\le i \le d$.

Let $m|d$ and write the decomposition $f= u(g)+h$ associated to $m$, 
where $f$ is viewed as a one variable polynomial in $x$ over $A[t]$.
We have
$$g(t,x) = x^{\frac dm} + b_1(t) x^{\frac dm - 1} + \cdots + b_{\frac dm}(t) \in A[t,x],$$
$$h(t,x) = \sum h_i(t) x^i,$$
and
$$u(t,x) = x^m + u_2(t) x^{m-2} + \cdots + u_m(t) \in A[t,x].$$

\begin{lemma}
\label{lem:deg}
 \  Under the assumptions and notation above, we have
\begin{enumerate}
  \item $\deg_x g = \frac dm, \deg_x u = m, \deg_x h < d - \frac dm$; 
  \item $\deg_t g \le \frac dm, \deg_t u \le d, \deg_t h \le d$;
  \item $\deg g = \frac dm, \deg u \le d, \deg h \le d$.
\end{enumerate}
\end{lemma}

\begin{proof}
The first item follows from the definition of the approximate $m$-root and the existence of such a decomposition.
We prove below a refinement of the second point. 

Fix some index $i$ with $1\leq i \leq d/m$. First we have $\deg b_i(t) \le i$: indeed from system (\ref{eq:sys})
we have $\deg b_1(t) = \deg a_1(t) \le 1$. Furthermore, 
$mb_i(t)$ is a $\Zz$-linear combination of $a_i(t)$ (which satisfies $\deg a_i(t) \le i$) and of terms $b_1^{j_1}\cdots b_{i-1}^{j_{i-1}}$
with $j_1+2j_2+\cdots+(i-1)j_{i-1}=i$. By induction we obtain $\deg b_i(t) \le i$.
This yields $\deg_t g \le \frac dm$ and $\deg_t g^j \le j\frac dm$.

If $\frac dm$ does not divide $i$, the coefficient $h_i(t)$ of $h(t,x) = \sum_{i=1}^d h_i(t)x^{d-i}$ 
is the coefficient of the highest monomial $\alpha_i(t)x^{d-i}$ in the difference between 
$f$ and powers of $g$. 

If $\frac dm$ divides $i$, let $j$ such that $i=j\frac dm$ and denote the former coefficient 
by $u_j(t)$ (it is the coefficient of $x^{m-j}$ in $u$). Then
$\deg u_j(t) = \deg \alpha_i(t) \le j \frac dm \le d$ ($j=2,\ldots, m$).
This implies that $\deg_t u_j(t)g^{m-j} \leq d$.

Conjoining the two cases, conclude that $\deg_t h \le d$.

This gives the second item and $\deg g= \frac dm$, $\deg h \le d$.
As $u$ is the sum of terms $u_j(t)x^{m-j}$ with $\deg u_j(t) \le j \frac dm$,
we have $\deg u \le \max_{j=2,\ldots,m}\big( j \frac dm + (m-j)\big) \le d$.
\end{proof}


\subsection{The Bertini-Noether conclusion}
If $\sigma: A\rightarrow B$ is a ring morphism, we denote  the image of elements $a\in A$ by $a^\sigma$. For $p(\underline x)\in A[\underline x]$, we denote the polynomial obtained by applying $\sigma$ to the coefficients of $p$ by  $p^\sigma(\underline x)$. 

\subsubsection{General statement} Fix an integral domain $A$ with quotient field $K$.

\begin{proposition} \label{prop:morphism} 
Let $f(x)\in A[x]$ be indecomposable in $K[x]$ of degree $d$ prime to the characteristic $p\geq 0$ of $K$. Then there exists a non-zero element $I_f\in A$ such that the following holds. For every morphism $\sigma:A\rightarrow k$ in a field $k$, if $I_f^\sigma\not=0$, then $f^\sigma(x)$ is indecomposable in $\overline k[x]$.
\end{proposition}

\begin{proof} Let $a_0$ be the coefficient of $x^d$ in $f(x)$, $\gamma =d \hskip 1pt a_0$ and $A_{\gamma^\infty}$ be the lo\-ca\-li\-zed ring of $A$ by the powers of $\gamma$. The polynomial $f(x)/a_0$ is in $A_{a_0^\infty}[x]$, is monic and is indecomposable in $K[x]$. For each non-trivial divisor $m$ of $d$, let 
$$f(x)/a_0 = u_m(g_m(x)) + h_m(x)$$ 
be the $m$-decomposition of $f(x)/a_0$ in $A_{\gamma^\infty}[x]$ ($m$ is invertible in $A_{\gamma^\infty}$). Each polynomial $h_m(x)$ writes $h_m(x) = h_m^A(x)/\gamma^{\nu_m}$ for some $h_m^A(x)\in A[x]$ and $\nu_m \in \Nn$, and is non-zero 
(as $f$ is indecomposable in $K[x]$). Let $h_{m0}$ be the (non-zero) coefficient of $h_m^A(x)$ of highest degree and set $I_f = \gamma \hskip 2pt \prod_m h_{m0}$. Consider next a morphism $\sigma:A\rightarrow k$ in a field $k$ such that 
$I_f^\sigma\not=0$. This morphism uniquely extends to some morphism $A_{\gamma^\infty}\rightarrow 
k$,  still denoted by $\sigma$. It is easily checked that the $m$-decomposition of 
$(f/a_0)^\sigma(x)$ in $k[x]$ is 
$$(f/a_0)^\sigma(x) = u_m^\sigma (g_m^\sigma(x)) + h_m^\sigma(x)$$
As $h_m^\sigma\not=0$ for all $m$, $(f/a_0)^\sigma$, and so also $f^\sigma$, is indecomposable in $k[x]$. \end{proof}

\subsubsection{Examples}

(a) For $A=\Zz$, then $I_f\in \Zz$, $I_f\not=0$.  
Proposition \ref{prop:morphism}, applied with $\sigma:\Zz\rightarrow \Ff_p$ the reduction 
morphism modulo a prime number $p$, yields the following:

\smallskip

\noindent
{\it for all suitably large $p$, the reduced polynomial $\overline f(x)$
modulo $p$ is indecomposable in $\overline{\Ff_p}[x]$.}
\smallskip

\noindent
This example will be refined in section \ref{sec:reduction}.

\medskip

\noindent
(b) Take $A=k[\underline t]$ with $k$ a field and $\underline t = (t_1,\ldots,t_r)$ some indeterminates. Denote in this situation by  $f(\underline t, x)$ the polynomial $f(x)$ of proposition \ref{prop:morphism}. Assume that $\deg(f)$ is prime to the characteristic of $k$ and that $f(\underline t, x)$ is indecomposable in $k(\underline t)[x]$. Proposition \ref{prop:morphism}, applied with $\sigma$ the specialization morphism $k[\underline t]\rightarrow \overline k$ that maps $\underline t = (t_1,\ldots,t_r)$ to an $r$-tuple $\underline t^\ast=
(t_1^\ast,\ldots,t_r^\ast)\in \overline k^r$ yields the following:
\smallskip

\noindent
{\it for all $\underline t^\ast$ off a proper Zariski closed subset of $\overline k^r$, the specialized polynomial  $f(\underline t^\ast, x)$ is indecomposable in $\overline k[x]$.}
\smallskip

\noindent
This example will be refined in section \ref{sec:proof_main_theorem}.
\medskip

\noindent
(c) Let $f(x)=x^d+a_1x^{d-1}+\cdots + a_d$ be the generic polynomial of degree $d\geq 1$ in one variable. Take for $A$ the ring $\Zz[\underline a]$ generated by the $d$-tuple of indeterminates $\underline a = (a_1,\ldots,a_d)$ corresponding to the coefficients of $f(x)$. 
The argument below shows that $f(x)$ is  indecom\-po\-sa\-ble in $\Qq(\underline a)[x]$.
Proposition \ref{prop:morphism}, applied next with $\sigma:A\rightarrow k$ a specialization morphism of $\underline a$ and $k$ any field of characteristic $0$, yields that all degree $d$ polynomials in $k[x]$ 
are indecomposable but possibly those from the proper Zariski closed subset corresponding to
the equation $I_f = 0$ (with $I_f$ viewed in $k[\underline a]$).

To show that $f(x)$ is indecom\-po\-sa\-ble in $\Qq(\underline a)[x]$, assume $f(x) = u(g(x))$ with 
$u,g\in \Qq(\underline a)[x]$ of degree $\geq 2$. As $f$ is monic, such a decomposition would 
exist with $u$ and $g$ monic in $\Qq[\underline a][x]$; this follows from lemma \ref{prop:prelim} below.
But then by specializing $\underline a$, it could be concluded 
that all degree $d$ polynomials in $\Qq[x]$ are decomposable. This is not the case, as for example corollary \ref{cor:intro} shows.


\section{Proof of theorem \ref{thm:reduction}} \label{sec:reduction} 

The proof is somewhat similar to the proof of lemma~\ref{lem:deg}.

Let $f \in \Zz[x]$ of degree $d$, $m$ be a divisor of $d$ and $p > d$ be a prime 
number. As in the proof of proposition \ref{prop:morphism}, we reduce to the case that $f$ is 
monic by dividing $f(x)$ by the leading coefficient $a_0$ and viewing the resulting 
polynomial in $\Zz_{a_0^\infty}[x]$. Then the reduction modulo $p$ of the $m$-decomposition 
$f=u(g)+h$ and the $m$-decomposition of the reduced polynomial modulo $p$ both exist 
and they coincide.

We say that a polynomial $p(x) = p_0x^d+ p_1x^{d-1} +\cdots + p_d$ of degree $\leq d$
is \emph{$f$-tame of order $d$} if there exist constants $\gamma_{i,d}$ such that 
$|p_i| \le \gamma_{i,d}\| \hskip 2pt f \|_\infty^i$ for all $i=0,\ldots,d$. This definition  
depends on $\|f\|_\infty$ and not on $\|p\|_\infty$.
Of course $f$ is itself $f$-tame of order $d$.

Using the system (\ref{eq:sys}), it follows by induction on $i$ that
$|b_i| \le \gamma_{i,d} \| f \|_\infty^i$, $i=1,\ldots, {\frac dm}$; thus $g$ is 
$f$-tame of order $\frac dm$.
Recall now how the decomposition is continued.
If $\frac dm$ does not divide $i$, the coefficient $h_i$ of $h = \sum_{i=1}^d h_ix^{d-i}$ 
is the coefficient of the highest monomial $\alpha_i x^{d-i}$ in the difference between $f$ and powers of $g$. 

If $\frac dm$ divides $i$, say $i=j\frac dm$, then pick the coefficient $\alpha_i$ of the highest 
monomial above and set $u_j = \alpha_i$: this is the coefficient of $x^{m-j}$ in $u(x)$. Deduce that
$|u_j| = |\alpha_i| \le \gamma_{j} \|f\|_\infty^{j \frac dm} \le \gamma \|f\|_\infty^d$ ($j=2,\ldots, m$)
for some constants $\gamma_{j}, \gamma$.

This implies that $u_j g^{m-j}$ is $f$-tame of order $d$ (even if it is a polynomial of degree $<d$).
Whence $h$ is $f$-tame of order $d$ and so $\| h \|_\infty \le \gamma_d \| f \|_\infty^d$.

Conclusion: as $f$ is not $m$-decomposable in $\Qq[x]$, $h(x) \neq 0$.
If $p > \gamma_d \| f\|_\infty^d$ then $h(x) \pmod p \neq 0$ and so
$f(x) \pmod p$ is not $m$-decomposable in $\overline{\Ff_p}[x]$.


\section{Proof of theorem \ref{th:main}} \label{sec:proof_main_theorem}

First note that assertions (b) and (c) immediately follow from 
assertion (a) and proposition \ref{prop:morphism}. We are left with proving assertion (a).
With no loss of generality we may assume that $\deg_{t_i}f\geq 0$. And this, due to the assumption 
on the characteristic of $k$, amounts to $\partial f/\partial t_i \not=0$, $i=1,\ldots,r$. Also recall that
due to the assumption on the characteristic of $k$, the polynomial $f(\underline t, x)$ is 
indecomposable in $\overline k[\underline t, x]$ \cite[theorem 4.2]{BDN}.

We divide the proof into two stages. 

\subsection{Stage 1: from $r$ to $2$ variables} \label{ssec:from_r_to_2}
Here we show that for $r\geq 2$, the polynomial $f(t_1+\alpha_1x,\ldots,t_{r-1}+\alpha_{r-1}x, t_r, x)$ is indecomposable in the polynomial ring $\overline{k(\alpha_1,\ldots,\alpha_{r-1}, t_1,\ldots,t_{r-1})}[t_r, x]$.

For this stage we use the following characterization: if $\underline y$ is a tuple of at least two indeterminates and $L$ an algebraically closed field, a polynomial $f(\underline y) \in L[\underline y]$ is indecomposable in $L[\underline y]$ if and only if $f(\underline y)-T$ is irreducible in $\overline {L(T)}[\underline y]$ (where $T$ is a new indeterminate). The desired conclusion readily follows by induction from the following result, which as explained in \cite[\S 2]{Na}, is a reformulation of the
Matsusaka-Zariski theorem \cite[proposition 10.5.2]{FrJa}.

\begin{proposition} 
\label{prop:matsusaka-zariski}
Let $s\geq 3$ be an integer, $\underline x = (x_1,\ldots,x_s)$ be an $s$-tuple of indeterminates and $Q(\underline x)\in k[\underline x]$ be an absolutely irreducible polynomial. Assume that $\partial Q/\partial x_1 \not=0$. Then if $\alpha_1$ is a new indeterminate, the polynomial $Q(x_1+\alpha_1x_s,x_2,\ldots,x_s)$ is
irreducible in $\overline{k(\alpha_1,x_1)}[x_2,\ldots,x_s]$.
\end{proposition}
\medskip

\subsection{Stage 2: from two to one variable} \label{ssec:from_2_to_1}
Here we show that for $r\geq 1$, $f(t_1+\alpha_1x,\ldots,t_{r}+\alpha_{r}x, x)$ is indecomposable in 
$\overline{k(\alpha_1,\ldots,\alpha_{r}, t_1,\ldots,t_{r})}[x]$.
From stage 1, we are reduced to proving the special case $r=1$ of theorem \ref{th:main} (a), which we restate below.

\begin{theorem} 
\label{th:main_r=1}
Let $f(t, x)$ be an indecomposable polynomial in $\overline k[t, x]$ with $k$ a field of characteristic $p=0$ or $p> \deg(f)$. Then the polynomial $f(t+\alpha x, x)$ is 
indecomposable in $\overline{k(\alpha,t)}[x]$. 
\end{theorem}
Again because of the assumption on the characteristic of $k$, $f$ could equivalently be assumed to be indecomposable in $k[t,x]$.

\begin{remark} \label{rem:contre-exemple}
The following example, inspired by \cite[p.~21]{Sc}, shows the conclusion fails if the assumption on the characteristic $p$ is removed. Take $k=\Ff_p$ and $f(t,x)=x^{p^2}+x^p+t$. As $\deg_t(f)=1$, $f(t,x)$ is indecomposable in $\overline k[t,x]$. But the polynomial $f(t+\alpha x, x) = x^{p^2}+x^p+t+\alpha x$ is decomposable in $\overline{k(\alpha,t)}[x]$ (and even in $\overline{k(\alpha)}[x]$): indeed, if $a, b\in \overline{k(\alpha)}$ satisfy $a + b^p=1$ and $a b = \alpha$, then we have $x^{p^2}+x^p+t+\alpha x = (x^p+b x)^p + a (x^p+b x) + t $.
\end{remark}

\subsubsection{Preliminary lemmas}
The following three lemmas will be used in the proof of theorem \ref{th:main_r=1}. The first one is due 
to Lecerf and Galligo \cite{GalLec}. It expresses in a simple and algebraic way a result already obtained by J.A.~Wood \cite{Wo}. We denote partial derivatives $\frac{\partial }{\partial \alpha}$
and $\frac{\partial }{\partial t}$ by 
$\partial_{\alpha}$ and $\partial_t$.

\begin{lemma}[Burger's equation lemma]
\label{lem:burger}
Let $k$ be a field and $f \in k[t,x]$ be a polynomial of degree $d$.  
Let $q(\alpha,t,x) = f(t+\alpha x,x) \in k[\alpha,t,x]$. Suppose $\phi \in \overline{k(\alpha,t)}$ 
is a simple root in $x$ of the polynomial $q(\alpha,t,x)$, i.e.,
$q(\alpha,t,\phi) = 0$ and $\displaystyle \partial_xq(\alpha,t,\phi) \neq 0$.
Then the derivations $\partial_{\alpha}$ and $\partial_t$ of $k(\alpha,t)$ uniquely extend 
to  $k(\alpha,t,\phi)$ and we have $\partial_{\alpha} \phi = \phi \cdot \partial_t \phi$.
\end{lemma}

\begin{proof}
Condition $\displaystyle \partial_xq(\alpha,t,\phi) \neq 0$ guarantees that $\partial_{\alpha}$ and $\partial_t$ uniquely extend to $k(\alpha,t,\phi)$.
Differentiate then $q(\alpha,t,\phi)=0$ with respect to $\alpha$ and with respect to $t$. Using next the special form 
$q(\alpha,t,x) = f(t+\alpha x,x)$ of $q$, this leads to the following formulas:
$$\left\{\begin{matrix}
& \partial_x q (\alpha ,t,\phi) \hskip 2pt \partial_{\alpha} \phi = - \partial_{\alpha} q(\alpha,t,\phi) = - \phi \hskip 2pt \partial_t f (\alpha+t\phi,\phi) \hfill\\
& \partial_x q (\alpha,t,\phi) \hskip 2pt \partial_t \phi = - \partial_t q(\alpha,t,\phi) = - \partial_t f (\alpha+t\phi,\phi) \hfill \\
\end{matrix}
\right.$$ 

\noindent 
which yields what we want.
\end{proof}

\begin{lemma}
\label{lem:simple}
Let $K$ be a field and $g \in K[v]$ be a polynomial such that $d=\deg(g)$
is prime to the characteristic of $K$.  For all but at most $d-1$ values $c\in K$, 
the polynomial $g(v) + c$ has only simple roots in $\overline{K}$.
\end{lemma}

\begin{proof}
Let $b_0\in K$ be the coefficient of $v^{d}$ in $g$. The discriminant of $g+c$ is 
$$\Delta = \mathrm{Res}(g+c,g^\prime) = d^{d} \hskip 1pt b_0^{2d-1} \hskip 2pt\prod_{\nu} (g(\nu)+c)$$
where in the product $\nu$ ranges over all roots $\nu\in\overline{K}$ of 
$g^\prime$ (with repetition for multiple roots). 
If $c$ is distinct from the $d-1$ values $-g(\nu)$ 
then $\Delta\not=0$ and
$g(v)+c$ have only simple roots.
\end{proof}

\begin{lemma}[Dujella-Gusic] 
\label{prop:prelim}
Let  $A$ be an integrally closed domain of quotient field $K$ and $f \in A[x]$, 
monic in $x$. If $f$ is decomposable in $K[x]$ then $f$ admits a decomposition in $A[x]$, i.e. 
$f = u(g)$ with $u,g \in A[x]$ monic in $x$.
\end{lemma}

\begin{proof}
See \cite[theorem 2.1]{DG} or \cite[theorem 2.1]{Gu}. The assumption on the characteristic that is made 
there is used to first reduce from decomposability on an extension of $K$ to decomposability over $K$ itself. 
This is not needed here as we assume $f$ decomposable in $K[x]$.
\end{proof}

\subsubsection{Proof of theorem \ref{th:main_r=1}}

We assume that $f(t+\alpha x,x) \in k[\alpha,t,x]$ is decomposable in $\overline{k(\alpha,t)}[x]$, and equivalently in $k(\alpha,t)[x]$, and we will prove that $f(t,x)$ is decomposable in $\overline k[t,x]$.

Adding a constant $c\in k$ to $f(x,y)$ changes $f(t+\alpha x,x)$ to $f(t+\alpha x,x) + c$ and does not affect the 
decomposability assumption nor the desired conclusion. Note next that $\deg_x(f(\alpha x,x)) = d$. As $p=0$ or $p>d$, it follows from lemma~\ref{lem:simple} that some element $c\in k$ can be found such that the polynomial $f(\alpha x,x)+c$ has only simple roots in $\overline{k(\alpha)}$. Up to replacing $f$ by $f+c$ we may and will assume 
that this is the case for $f(\alpha x,x)$ itself.

If $f_d(t,x)\in k[t,x]$ denotes the homogeneous part of degree $d$ in $f(t,x)$, the leading coefficient of $f(\alpha x,x)$, relative to $x$, is $f_d(\alpha ,1)$.  Consider now the polynomial $\tilde{q}(\alpha,t,x)=f(t+\alpha x,x)/f_d(\alpha ,1)=q(\alpha ,t,x)/f_d(\alpha ,1)$.
By construction $\tilde{q} \in k(\alpha)[t][x]$, is monic in $x$ and is decomposable in $k(\alpha,t)[x]$. 
By lemma \ref{prop:prelim} applied with $A = k(\alpha)[t]$,
we get $\tilde{q}=u(g)$ with $u, g \in k(\alpha)[t][x]$, monic in $x$ and such that 
$\deg_x u=m \geq 2$ and $\deg_x g=d/m \geq 2$. Set

$$\tilde{q}(\alpha,t,x) = \prod_{i=1}^d (x-\phi_i) = \prod_{j=1}^{m} (g(\alpha,t,x)-\lambda_j),$$

\noindent
so that $\phi_1,\ldots,\phi_d \in \overline{k(\alpha,t)}$ are the roots of $\tilde{q}$,
and $\lambda_1,\dots,\lambda_m \in \overline{k(\alpha,t)}$ are the roots of $u$.
Furthermore, by uniqueness of factorization, there exists a partition of $\{1,\ldots,d\}$
into subsets $I_1,\ldots,I_m$ of $\{1,\ldots,d\}$ such that:

$$\prod_{i \in I_j} (x-\phi_i) = g(\alpha,t,x) - \lambda_j \  \ (j=1,\ldots,m)$$

We will use Newton's identities: for a polynomial $p(x) = x^n+p_1x^{n-1}+\cdots +p_{n-1}x+p_n = \prod_{i=1}^n (x-\phi_i)$, setting $S_\ell = \sum_{i=1}^n \phi_i^\ell$, we have:
\begin{equation}
\label{eq:newton}
\tag{N} 
 S_\ell +p_1 S_{\ell-1}+\cdots+p_{\ell-1} S_1+\ell p_{\ell} =0 \quad (\ell = 1,\ldots, n)
\end{equation}

\noindent
Applied to $g(\alpha,t,x)-\lambda_j$ (for which only the constant term depends on $j$), this provides the following: for every $\ell=1,\ldots, \frac d m -1$ and $j=1,\ldots,m$,
\begin{equation}
\sum_{i\in I_1} \phi_i^\ell = \sum_ {i\in I_j} \phi_i^\ell  \in k(\alpha)[t].   
\label{eq:phil}
\tag{*}
\end{equation}

At this stage we use our initial reduction to the situation that $f(\alpha x,x)$ has only simple roots in $\overline{k(\alpha)}$. This implies first that $\tilde{q}(\alpha ,t,x)$ has only simple roots in $\overline{k(\alpha,t)}$, and, second, that these roots, $\phi_1,\ldots,\phi_d$, can be viewed in the ring $\overline{k(\alpha)}[[t]]$ of formal power series in $t$ with coefficients in $\overline{k(\alpha)}$, {\it via} some embedding $k(\alpha,t)(\phi_1,\ldots,\phi_d)\subset \overline{k(\alpha)}((t))$; such an embedding indeed exists thanks to Hensel's lemma.

Differentiation of (\ref{eq:phil}) for $\ell = \frac d m -1$ with respect to $\alpha$ then provides

$$\sum_{i\in I_1} (\frac dm -1) \cdot \partial_{\alpha} \phi_i \cdot  \phi_i^{\frac d m -2} = \sum_{i\in I_j} (\frac dm -1) \cdot \partial_{\alpha} \phi_i \cdot \phi_i^{\frac d m -2}  \in k(\alpha)[t].$$   

\noindent
Use lemma \ref{lem:burger} to deduce that

$$\sum_{i\in I_1}  \phi_i \cdot \partial_t \phi_i \cdot  \phi_i^{\frac d m -2} = \sum_{i\in I_j} \phi_i \cdot \partial_t \phi_i \cdot \phi_i^{\frac d m -2}  \in 
k(a)[t]$$    

\noindent 
and to conclude that
\begin{equation}
\partial_t \left( \sum_{i\in I_1}     \phi_i^{\frac d m}\right)  = \partial_t \left(\sum_{i\in I_j}  \phi_i^{\frac d m} \right) \in k(\alpha)[t]  \hskip 10mm (j=1,\ldots,m)  
\label{eq:phil2}
\tag{**}
\end{equation}

\noindent
Use this conclusion for $j=1$ to write $\sum_{i\in I_1} \phi_i^{\frac d m} = P_1 + d_1$ for some $P_1 \in k(\alpha)[t]$ with $P_1(\alpha,0)=0$ and some $d_1 \in \overline{k(\alpha)}[[t^p]]$, and to deduce next that $\sum_{i\in I_j} \phi_i^{\frac d m}=P_1+d_j$ for some $d_j \in \overline{k(\alpha)}[[t^p]]$, $j=1,\ldots, m$.

\begin{remark} If the characteristic is $p=0$, then the elements $d_1,\ldots,d_m$ are 
constants in $\overline{k(\alpha)}$, and the end of the proof below is simpler.
\end{remark}

The Newton identity (\ref{eq:newton}) with $\ell=d/m$ and $p(x)=g(\alpha,t,x)-\lambda_j$ gives

$$\prod_{i\in I_j} (-\phi_i) = g_{d/m} = -\frac m d (S_{d/m} + g_1 S_{d/m-1} + \cdots +g_{d/m-1}S_1).$$
\noindent
where $g_1,\ldots,g_{d/m}\in k[\alpha,t]$ are the coefficients of $g$ with respect to $x$. From display (\ref{eq:phil}), the sums $S_1,\ldots,S_{\frac dm - 1}$ lie in    
 $k(\alpha)[t]$ 
and are independent of $j=1,\ldots,m$. And from above we have $S_{\frac dm} =  P_1 + d_j$. Therefore there exists a polynomial $P_0 \in k(\alpha)[t]$
(independent of $j$) and elements 
$e_1,\ldots, e_m \in \overline{k(\alpha)}[[t^p]]$ such that $\prod_{I_j} (-\phi_i) = P_0 + e_j$ ($j=1,\ldots,m$).
This provides this formula for $\lambda_j$ ($j=1,\ldots,m$):

$$\lambda_j = g(\alpha,t,0) - \prod_{i\in I_j} (-\phi_i) = g(\alpha,t,0) - P_0 - e_j$$

Set $G(\alpha,t,x) = g(\alpha,t,x)-g(\alpha,t,0)+ P_0 \in k(\alpha)[t,x]$ so that 
$$\tilde{q}(\alpha,t,x)= \prod_{j=1}^{m} (g(\alpha,t,x)-\lambda_j) = \prod_{j=1}^{m} (G(\alpha,t,x)+e_j).$$
This provides the decomposition $\tilde{q} = v(G)$ with $v(x) = \prod_{j=1}^{m} (x+e_j)$ in $\overline{k(\alpha)}[[t^p]][x]$ and $G \in k(\alpha)[t,x]$.

As $\tilde{q}$ and $G$ lie in $k(\alpha)[t,x]$ we deduce that $v \in k(\alpha,t)[x]$: indeed, once we know $\tilde{q}$ and $G$, the computation of $v$ is reduced to the resolution of a linear system. But then by lemma \ref{prop:prelim} one may take $v \in k(\alpha)[t,x]$.
Up to a linear change of variables $x \mapsto x - a$, one may also assume that $v(x)$ is of the form $v(x) = x^m + v_{2}x^{m-2}+\cdots$ ({\it i.e.} $v_1=0$), so that we can apply lemma \ref{lem:deg}.
Conclude that $\deg_t v \leq d$. As $v\in \overline{k(\alpha)}[[t^p]][x]$ and $p>d$ we deduce that $v \in k(\alpha)[x]$. This shows that $\tilde{q}$ is decomposable in $k(\alpha)[t,x]$.

Multiply the equality $\tilde{q}=v(G)$ by $f_d(\alpha,1)$ to get that $q$ is decomposable in $k(\alpha)[t,x]$, that is: $q(\alpha,t,x)= f(t+\alpha x,x)= v^\prime(\alpha, (G^\prime(\alpha,t,x))$ with $v^\prime \in k(\alpha)[x]$ of degree $\geq 2$ and $G^\prime(\alpha,t,x) \in k(\alpha)[t,x]$. For all but finitely many $\alpha^\ast\in \overline k$, specialization of $\alpha$ to $\alpha^\ast$ of this decomposition provides the non-trivial decomposition $f(t+\alpha^\ast x,x)= v^\prime(\alpha^\ast, (G^\prime(\alpha^\ast,t,x))$. But then the change of variables 
$(t,x) \mapsto (t-\alpha^*x,x)$ shows that $f(t,x)$ is decomposable in $\overline k[t,x]$.

\subsection{Explicit versions}
We explain here how our method can be used to get explicit results. For simplicity, we restrict to polynomials in two variables.

\begin{corollary} \label{thm:optimal_bound_new}
Let $f(t, x)$ be an indecomposable polynomial in $k[t, x]$ with degree $d$
where $k$ is a field of characteristic $p=0$ or $p> d$. 
Then there exist polynomials $h_{m,i}(\alpha,t) \in k[\alpha,t]$ where $m|d$, and $i=1,\ldots,d-d/m$ of total degree $\leq md^2+2d$ with the following property: for all $(t^{\ast}, \alpha^{\ast}) \in k^2$, 
if for each divisor $m$ of $d$ there exists $i_0$ such that  $h_{m,i_0}(\alpha^{\ast},t^{\ast}) \neq 0$, 
then $f(t^{\ast}+\alpha^{\ast}x,x)$ is indecomposable of degree $d$ in $k[x]$.
\end{corollary}

\begin{proof}
The proof is a variation of that of lemma \ref{lem:deg} or of theorem \ref{thm:reduction}.
Set 
$$q(\alpha,t,x) = f(t+\alpha x,x) = a^\prime_0(\alpha,t)x^d+a^\prime_1(\alpha,t)x^{d-1}+\cdots+ a^\prime_d(\alpha,t)$$
Due to the assumption  $\deg f =d$ we get:
$$\deg_t a^\prime_i(\alpha,t) \le i, \quad \deg_\alpha a^\prime_i(\alpha,t) \le d-i, \quad \deg a^\prime_i(\alpha,t) \le d; \quad i=0,\ldots,d.$$
In particular $a^\prime_0(\alpha,t) = a^\prime_0(\alpha)$ does not depend on $t$.

Consider then $\tilde{q}(\alpha,t,x) = q(\alpha,t,x)/a^\prime_0(\alpha)$; this is a polynomial in $k(\alpha)[t,x]$, monic in $x$.
Consider the $m$-decomposition of $\tilde{q}$ with respect to the variable $x$:
$\tilde{q} = u_m(g_m)+h_m$.

\smallskip

\emph{Variable $t$.} Apply lemma \ref{lem:deg} to the polynomial $\tilde{q}$ seen as a polynomial
in $A[t,x]$ with $A = k(\alpha)$, since $\deg_t a^\prime_i(\alpha,t) \le i$. This yields  
$\deg_t h \le d$.

\smallskip

\emph{Variable $\alpha$.} Consider now $\tilde{q}$ as a rational fraction in $\alpha$ and as a polynomial in $x$
 to compute the degree in $\alpha$ of $h$ (we forget the variable $t$). 
Lemma \ref{lem:deg} cannot be applied since the degree of the coefficients does not
satisfy the correct hypothesis (moreover the coefficients are not polynomials in $\alpha$).
The $m$-decomposition $\tilde{q}= u_m(g_m)+h_m$ lives in $A[\alpha]_{(a^\prime_0(\alpha))^\infty}[x]$ with $A=k[t]$.

A polynomial $g(\alpha,x)=c_0x^\delta+c_1(\alpha)x^{\delta-1}+\cdots+c_\delta(\alpha)$ in $A[\alpha]_{(a^\prime_0(\alpha))^\infty}[x]$ is
\emph{$\alpha$-tame of order $\delta$} if each $c_i(\alpha)$ can be written ($i=0,\ldots,\delta$):
$$c_i(\alpha) = \frac{c^\prime_i(\alpha)}{a^\prime_0(\alpha)^i} \text{ with } \deg c^\prime_i(\alpha) \le i\delta.$$

Note that $a^\prime_0(\alpha)$ comes from $\tilde{q}$ and is fixed.

The following properties can easily be proved:
\begin{enumerate}
  \item $\tilde{q}(\alpha,t,x)$ is $\alpha$-tame of order $d$. 
  \item The sum of two $\alpha$-tame polynomials of order $\delta$ is $\alpha$-tame of order $\delta$. 
  \item The product of a $\alpha$-tame polynomial of order $\delta$ and a $\alpha$-tame polynomial of order $\delta^\prime$
is $\alpha$-tame polynomial of order $\delta+\delta^\prime$. 
  \item \label{item:4} The $k$-power a $\alpha$-tame polynomial of order $\delta$ is $\alpha$-tame of order $k\delta$. 
  \item \label{item:5} An $\alpha$-tame polynomial of order $jd$ is an $\alpha$-tame polynomial of order $md$ ($j=1,\ldots,m$). 
\end{enumerate}

By inspection of system $(\mathcal{S})$, we have in the decomposition
$\tilde{q}= u_m(g_m)+h_m$, that $g_m$ is $\alpha$-tame of order $d$.
The proof is very similar to the one in lemma \ref{lem:deg}.
Then by item \ref{item:4}, $g_m^j$ are $\alpha$-tame of order $jd$, and by item
\ref{item:5}, $\tilde q$ and $g_m^j$ are  $\alpha$-tame of order $md$, ($j=1,\ldots,m$).
As in lemma \ref{lem:deg} we distinguish two cases:

If $\frac dm$ does not divide $i$, the coefficient $h_{m,i}(\alpha)$ of $h_m(x) = \sum_{i=1}^d h_{m,i}(\alpha)x^{d-i}$ 
is the coefficient of the highest monomial $\gamma_i(\alpha)x^{d-i}$ in the difference between 
$\tilde{q}$ and powers of $g_m$. 

If $\frac dm$ divides $i$, let $j$ such that $i=j\frac dm$ and denote the former coefficient 
by $u_{m,j}(\alpha)$ (it is the coefficient of $x^{m-j}$ in $u_m$). 
Then $u_{m,j}(\alpha) = \gamma_i(\alpha) = \frac{\gamma^\prime_i(\alpha)}{a^\prime_0(\alpha)^i}$.
This implies that $u_{m,j}(\alpha)g_m^{m-j}$ is $\alpha$-tame of order $md$.

Both cases imply that $u_m(g_m)$ and $h_m$ are $\alpha$-tame of order $md$.

\medskip

\noindent
\emph{Conclusion.}
The $m$-decomposition $\tilde{q}= u_m(g_m)+h_m$ provides a decomposition 
$q(\alpha,t,x)= a^\prime_0(\alpha) \times \frac{1}{a^\prime_0(\alpha)^d} (u^\prime_m(g^\prime_m)+h^\prime_m)$
with $h^\prime_m = \sum h^\prime_{m,i}(\alpha,t)x^{d-\frac{d}{m}-i}$ a 
\emph{polynomial} in $k[\alpha,t,x]$ whose coefficients 
satisfy: $\deg_\alpha h^\prime_{m,i} \le md^2$ and $\deg_t h^\prime_{m,i} \le d$.
Hence $\deg h^\prime_{m,i} \le md^2+d$. Finally, if $a^\prime_0(\alpha)\neq 0$ then as usual $q$ is $m$-decomposable if and only if $h^\prime_{m,i}=0$ for all $i$.\\
We set $h_{m,i}(\alpha,t)=a^\prime_0(\alpha) \hskip 2pt h^\prime_{m,i}(\alpha,t)$ and we have the desired result.
\end{proof}

\begin{corollary}
Let $f(t, x)$ be an indecomposable polynomial in $k[t, x]$ with degree $d$
where $k$ is a field of characteristic $p=0$ or $p> d$. 
Let $S$ be a finite subset of $k$. For a uniform random choice of $\alpha$, $t$ in $S$, the probability
$$\mathcal{P}\big( \{ f(t^{\ast}+\alpha^{\ast}x,x) \textrm{ is indecomposable in } k[x] \, |\,  \alpha, t \in S\} \big)$$
is at least equal to $1 -\mathcal{D}/|S|$, with $\mathcal{D} =\sigma_1(d).d^2+2\sigma_0(d).d$ where $\sigma_1(d)=\sum_{m|d}m$, $\sigma_0(d)$ is the number of divisors of $d$  and  $|S|$ is the cardinal of $S$.
\end{corollary}

\begin{proof}
$f(t^{\ast}+\alpha^{\ast}x,x)$ is indecomposable in $k[x]$, if for all $m|d$ there exists $i_0$ such that $h_{m,i_0}(\alpha^{\ast},t^{\ast})=0$. Moreover $\bigcup_{m|d}\bigcap_i \{h_{m,i}(\alpha^{\ast},t^{\ast})=0\}$ is a subset of $\{\prod_{m|d}h_{m,i_0}(\alpha^{\ast},t^{\ast})=0\}$. As $\deg(h_{m,i_0}) \leq md^2+2d$, by Zippel-Schwartz's lemma, see e.g. \cite[lemma 6.44]{GaGe}, we have the desired result.
\end{proof}


\end{document}